\documentclass[12pt]{amsart}
 
\usepackage[utf8]{inputenc} 
\usepackage[english]{babel}
\usepackage{amsthm}
\usepackage{amssymb}
\usepackage{amsfonts}
\usepackage[T1]{fontenc}
\usepackage{parskip}
\usepackage{geometry}
\usepackage{amsmath}
\usepackage{caption}
\usepackage{subcaption}
\usepackage{array}
\usepackage{float}
\usepackage{mathdots}
\usepackage{braket}
\usepackage{color}

\usepackage{comment}

\usepackage{hyperref}

\DeclareMathOperator{\lcm}{lcm}
 
\numberwithin{equation}{section} 
 
\theoremstyle{theorem} 
\newtheorem{theorem}{Theorem}[section]

\newtheorem{corollary}{Corollary}[section]
\newtheorem{lemma}{Lemma}[section]

\theoremstyle{definition}

\newtheorem{remark}[theorem]{Remark}

\title[Character groups of metacyclic groups]{Character groups of metacyclic groups} 
 
\author[J.A. Stevenson]{Justin A. Stevenson}
\address{Department of Mathematics\\
Iowa State University\\
Ames, Iowa 50011, U.S.A.}
\email{jas1@iastate.edu}
 
\begin{document}

\providecommand{\keywords}[1]
{
  \small	
  \textbf{\textit{Keywords---}} #1
}
 
\begin{abstract} 
Pontryagin duality provides a powerful tool for analyzing the structure and properties of locally compact abelian groups, in particular finite abelian groups. Now much can be done with non-abelian groups via the construction of character groups. In this paper we provide a set of conditions a group must satisfy to be realized as a balanced character group for a metacyclic group. Furthermore we provide a classification for the balanced character groups of a particular class of metacyclic groups, the dicyclic groups.
\end{abstract} 


\maketitle


\section{Introduction}

The dual group of a locally compact abelian topological group $A$ is defined as the set of continuous homomorphisms from $A$ to the circle group, the multiplicative group of complex numbers of unit modulus. In other words it is the collection $\widehat{A}$ of all complex-valued continuous characters on $A$. The group structure is given by pointwise multiplication of characters, with the identity being the trivial character sending everything to $1$. \emph{Pontryagin duality} states that the double dual group is naturally isomorphic to the original group. 

Now we are able to construct an analogue of the character group for not necessarily abelian finite groups. A construction was given in \cite{Smith2020} to endow the collection of irreducible characters of a finite group with the structure of a weighted quasigroup, which, as shown by Hilton in \cite{HiltonWojciechowski}, can be lifted to a quasigroup, a \emph{character quasigroup} of the finite group. Under this definition every finite abelian group has a unique lift, namely itself. This observation underlies our view that character quasigroups (especially when they are associative \emph{character groups}) provide a generalization of the dual group to the non-abelian case.

In a previous paper, the character groups of the dihedral and generalized quaternion groups were studied \cite{StevensonSmith}. The goal of the current paper is to develop the theory within a more general class of groups, the metacylic groups. A group is said to be \emph{metacyclic} if it has a cyclic normal subgroup such that the quotient group is also a cyclic group. Equivalently, a group $G$ is metacyclic if and only if there exist two cyclic subgroups $N$ and $K$, with $N$ normal in $G$, such that $G\cong NK$. 

Let $x\in G$ be such that $N=\braket{x}\cong \mathbb{Z}/_n$ with isomorphism given by $x^a\mapsto n\mathbb{Z} + a$, and let $G/N=\braket{yN}$ for some $y\in G$. Assuming that the order of $G/N$ is $m$, we have $y^m\in N$, so $y^m=x^l$ for some integer $l$. Additionally, since $N$ is normal, we must have $x^y=x^k$ for some integer $k$ such that $\gcd(n,k)=1$ and the congruences $l(k-1) \equiv 0\mod n$ and $k^m\equiv 1 \mod n$ hold. The second condition results from the following argument:
$$ x^{l(k-1)}=x^{lk}x^{-l}=(x^y)^lx^{-l}=(x^l)^yx^{-l}=(y^m)^yx^{-l}=y^mx^{-l}=1\,, $$ 
and the computation
$$ x = x^{-l}xx^l = y^{-m}xy^m = y^{-m}y^mx^{k^m} = x^{k^m}$$
justisfies the last condition.

It was shown by H\"{o}lder \cite{Zassenhaus} that a group $G$ is metacyclic if and only if it has a presentation 
\begin{equation}
G=\braket{x,y \ | \ y^m=x^l, x^n=1, x^y=x^k} 
\end{equation}
where $\gcd(n,k)=1$ and the congruences $k^m\equiv 1 \mod n$ and $l(k-1) \equiv 0 \mod n$ hold. In this case, we will denote the metacyclic group with parameters $k,l,m,n$ by $M(k,l,m,n)$. One must be careful though, as these parameters are not invariants for the group: It is possible for two metacyclic groups with different parameters to be isomorphic. Additionally, it is worth noting that $m$ need not be the smallest integer such that $k^m\equiv 1\mod n$. The metacyclic group
$$ \braket{x,y | x^3=1, y^4=1, x^y=x^2}$$
provides an example. It is not isomorphic to any metacyclic group with parameter $m=2$.

\section{Quasigroups, weighted quasigroups and character quasigroups}

\subsection{Quasigroups}

A \emph{quasigroup} is a pair $(Q,\cdot)$ where $Q$ is a set and $\cdot$ is a binary operation such that for all $a,b\in Q$, there exist unique solutions $x,y\in Q$ to the equations
\begin{equation}\label{E:lrdiveqn}
a\cdot x=b
\quad
\mbox{ and }
\quad
y\cdot a=b\,.
\end{equation}
Small finite quasigroups are conveniently displayed by their Cayley or multiplication tables. The above condition requires every element of $Q$ to occur in each row and column exactly one time. The body of the Cayley table of a finite quasigroup defines a Latin square, and conversely each Latin square is the body of the Cayley table of some finite quasigroup. 

This definition of a quasigroup is combinatorial. Algebraically speaking, we may define a quasigroup as a tuple $(Q,\cdot, / , \backslash)$ where $Q$ is a set and $\cdot, /,\backslash$ are all binary operations, called the \emph{multiplication}, \emph{left division} and \emph{right division} respectively, such that: 
\begin{align*}
a \cdot (a \backslash b) &= b\,; 
\quad(a/b)\cdot b = a\,;
\\
a\backslash(a\cdot b)&=b\,;
\quad(a\cdot b)/b = a\,.
\end{align*}
Using this notation, we may write the solutions to equations \eqref{E:lrdiveqn} as $x=a\backslash b$ and $y=a/b$. For further reading and background on the theory of quasigroups, see \cite{SmithBook}.

\subsection{Weighted quasigroups}

The definition of a quasigroup can be generalized by allowing the base set $Q$ to be a multiset. Concretely, we say that a \emph{weighted quasgroup} is a triple $(X,w,\alpha)$ where $X$ is a set, equipped with a \emph{weight function} $w:X\rightarrow \mathbb{N}$ and a \emph{multiplication function}
$$
\alpha\colon X\times X\times X\to\mathbb N;
(x,y,z)\mapsto\alpha_z(x,y)
$$ 
such that 
 \begin{equation}\label{weight}
 \sum_{z\in X} \alpha_z(x,y) = \sum_{z\in X} \alpha_x(z,y) = \sum_{z\in X} \alpha_y(x,z) = w(x)w(y)
 \end{equation}
for all $x,y\in X$.
This definition appears in \cite{Hilton,HiltonWojciechowski}. The multiplication function may also be displayed in a Cayley table format:
\begin{equation}\label{E:WtdQgpGn}
\begin{tabular}{c||ccc} 
$*$ &\dots&$y$&\dots\\
\hline	
\hline
$\vdots$&&$\vdots$&\\
$x$&\dots&$\sum_{z\in X}\alpha_z(x,y)z$&\dots\\
$\vdots$&&$\vdots$&\\
\end{tabular}
\end{equation}
We define the \emph{gross weight} of a weighted quasigroup $(X,w,\alpha)$ to be $\sum_{z\in X}w(z)$. A quasigroup $(Q,\cdot)$ of finite order $n$ may be regarded as a weighted quasigroup of gross weight $n$, with $w(x)=1$ and $\alpha(x,y,z)=\delta_{x\cdot y,z}$ for $x,y,z\in Q$. In this case, the conditions of \eqref{weight} amount to the Latin square conditions that each symbol from $Q$ appears exactly once in each row and column of the body of the multiplication table of $Q$.

\subsection{Covering weighted quasigroups with quasigroups}\label{SS:CoWeQuQu}

A quasigroup $Q$ of finite order $n$ is said to \emph{cover} a weighted quasigroup $(X,w,\alpha)$ of gross weight $n$ if there is a surjective function $f\colon Q\to X$ such that
$$
w(x)=\big|f^{-1}\Set{x}\big|
$$
and
$$
\alpha(x,y,z)=
\big|\Set{(a,b)\in f^{-1}\Set{x}\times f^{-1}\Set{y}|f(a\cdot b)=z}\big|
$$
for all $x,y,z\in X$. We will also say that $(X,w,\alpha)$ \emph{lifts} to $Q$ if such a covering exists. The function $f$ is called a \emph{covering} or \emph{lifting}.

The following theorem was presented in the given form as \cite[Th.~5.3]{Smith2020}.
\begin{theorem}
Each weighted quasigroup 
of gross weight $n$ lifts to a quasigroup 
of order $n$.
\end{theorem}
The original formulation appeared in \cite{HiltonSchool}, and then again in different language as \cite[Th.~1]{Hilton}, \cite[Th.~1]{HiltonWojciechowski}. The proofs were purely combinatorial, relying on regularity results from graph theory to build up a Latin square stepwise by row, column, and symbol.

\subsection{Weighted character quasigroups}\label{SS:WeChrQgp}

Suppose that $G$ is a group of finite order $n$, with set $\widetilde\Gamma(G)$ or $\widetilde\Gamma=\Set{\theta_1,\dots,\theta_t}$ of irreducible characters.\footnote{The notation $\widetilde\Gamma$ refers to the dual of the conjugacy class association scheme $\Gamma$ of $G$, as discussed in \cite[Exs.~2.2, 3.1]{Bannai} \cite[Ex.~2.1(2),\S2.5]{BanIto}, \cite[\S2.2.3]{Smith2020}.} Then a weighted quasigroup structure is defined on $\widetilde\Gamma$ by
\begin{equation}\label{E:weight}
w\colon\widetilde\Gamma\to\mathbb N;\theta_i \mapsto \theta_i(1)^2 
\end{equation}
and
\begin{equation}\label{def}
\alpha\colon\widetilde\Gamma^3\to\mathbb N;(\theta_i,\theta_j,\theta_k)\mapsto \theta_i(1)\theta_j(1)\theta_k(1)\braket{\theta_i\cdot\theta_j|\theta_k} 
\end{equation}
with scalar product
\begin{equation}\label{inner prod}
\braket{\theta|\varphi}=\frac1n\sum_{g\in G}\theta(g)\overline{\varphi(g)} 
\end{equation}
for functions $\theta,\varphi\colon G\to\mathbb C$
\cite[Cor.~5.10]{Smith2020}. The structure $(\widetilde\Gamma,w,\alpha)$ is called the \emph{weighted character quasigroup} of $G$. Its gross weight 
is $\sum_{\theta_i\in\widetilde\Gamma}\theta_i(1)^2=n$.

\subsection{Character quasigroups}

Specializing the conditions of \S\ref{SS:CoWeQuQu}, the weighted character quasigroup lifts to a quasigroup $Q$ if there exists a surjection $f\colon Q\to\widetilde{\Gamma}$ such that $\theta_i(1)^2=|f^{-1}\{\theta_i\}|$ for $\theta_i\in\widetilde{\Gamma}$ and 
\begin{equation}\label{alpha}
\alpha(\theta_i,\theta_j,\theta_l)=
\big|\Set{(a,b)\in f^{-1}\{\theta_i\}\times f^{-1}\{\theta_j\}|f(a\cdot b)=\theta_l}\big|
\end{equation}
for all $\theta_i,\theta_j,\theta_l\in\widetilde{\Gamma}$. A quasigroup cover $Q$ of $\big(\widetilde{\Gamma}(G),w, \alpha\big)$ is called a \emph{character quasigroup} of $G$, and a \emph{character group} of $G$ if the cover is itself a group.

By \cite[Th.~5.11]{Smith2020}, character quasigroups are guaranteed to exist for any finite group $G$. However, it still unknown if a character group always exists. The current paper seeks to answer this question in the affirmative for the metacyclic groups.

\section{Characters of metacyclic groups}

Recall that a metacyclic group is a group with the following presentation 
$$ M(k,l,m,n)=\braket{x,y \ | \ y^m=x^l, x^n=1, x^y=x^k}\,,$$
where where $\gcd(n,k)=1$ and the congruences $k^m\equiv 1 \mod n$ and $l(k-1) \equiv 0 \mod n$ hold. The character theory of the metacyclic groups is well known: see \cite[Theorem~5.1]{Munkholm}, for example. The group $K:=\braket{y}$ acts on $\mathbb{Z}/_n$ via $y\cdot (n\mathbb{Z}+1) = n\mathbb{Z}+k$, under which $y^m$ acts trivially. The orbit and stabilizer of an element $n\mathbb{Z}+a$ will be denoted as $Ka$ and $K_a$ respectively. Let $I$ be a set of distinct representatives for the orbits of this action. Note that the size of the orbit $Ka$ is equivalently the smallest integer such that $ak^{|Ka|} \equiv a \mod n$, which we will often denote by $r_a$. Similarly we will denote $|K_a|$ by $s_a$.  

For each $a\in I$ we let $\Set{q_{a,1},\dots, q_{a,s_a}}$ be a full set of integers, pairwise incongruent modulo $|K_a|$ with $-\frac{al}{n}\leq q_{a,\alpha}<m-\frac{al}{n}$ for integers $\alpha$ in the range $1\le\alpha\le s_a$, and set 
\begin{equation*}
e_{a,\alpha}=\frac{al+nq_{a,\alpha}}{\gcd(n,l)}\,.
\end{equation*}
For concreteness, we make the canonical choice of $$q_{a,\alpha}=\alpha-\lfloor {al}/{n} \rfloor\,,$$ and thus obtain the integers
\begin{equation}\label{E:ea,alpha}
e_{a,\alpha}=\frac{n\alpha+al-n\lfloor \frac{al}{n} \rfloor}{\gcd(n,l)}= \frac{n}{\gcd(n,l)}(\alpha+\{al/n \})
\end{equation}
for $a\in I$ and $1\le\alpha\le s_a$, where $\{x \}$ denotes the fractional part of $x$. In particular, the relation
\begin{equation}\label{E:frptalnZ}
\frac n{\gcd(n,l)}\bigg\{\frac{al}n\bigg\}\in\mathbb Z
\end{equation}
holds.

Next, let $\omega$ be a primitive $n$-th root of unity, and let $\eta$ be a primitive $\bigg(\dfrac{nm}{\gcd(n,l)}\bigg)$-th root of unity, such that $\omega^{\gcd(n,l)}=\eta^m$. The characters $\theta_{a,\alpha}$ are all induced from the linear characters $\chi_{a,\alpha}\colon G_a\rightarrow \mathbb{C}$ with $G_a=\braket{y^{r_a},x}$, where
$
\chi_{a,\alpha}(y^{r_a})=
\eta^{r_ae_{a,\alpha}}
$
and
$
\chi_{a,\alpha}(x)=\omega^a 
$
\cite{Munkholm}.
For each $a\in I$ and $0\leq \alpha < r_a$, the specification  
$$
 \theta_{a,\alpha}(y^jx^i)=
  \begin{cases}
\eta^{je_{a,\alpha}}\sum_{s\in Ka} \omega^{si} & \mbox{if } j=r_aj' \mbox{ for some } 0\leq j'<s_a \\
0 & \mbox{otherwise} \\
 \end{cases}
$$
defines the character of an irreducible representation of $G$ of dimension $r_a$.

\subsection{Products of characters}\label{SS:ProdChar}

For $a,b,c\in I$, define
$$
S(a,b,c):= |\Set{(r,s)\in Ka\times Kb :r+s\equiv c \mod n}| \,.
$$
The motivation for this definition of $S(a,b,c)$ is the following identity, which will appear in our calculation of the product of two characters:
\begin{equation}
 \sum_{i=0}^{n-1} \sum_{s\in Ka} \sum_{r\in Kb} \sum_{t\in Kc} \omega^{(s+r-t)i}=
 nr_c S(a,b,c) \,.
\end{equation}
The identity follows from the easily verified fact that $S(a,b,c)=S(a,b,ck^t)$ for any choice of $1\leq t\leq r_c$.

\begin{lemma}\label{RootOf1Lem}
Let $a,b,c$ in $I$ be such that $S(a,b,c)\neq 0$, and let $0\leq \alpha <r_c, 0\leq \beta<r_b$ and $0\leq \gamma <r_c$. Then the congruence
$$ 
e_{a,\alpha}+e_{b,\beta}-e_{c,\gamma} \equiv
 0 \mod \frac{n}{\gcd(n,l)} 
 $$
 holds, using the notation of \eqref{E:ea,alpha}.
\end{lemma}

\begin{proof}
Since $a,b,c\in I$ are chosen such that $S(a,b,c)\neq 0$, by definition we have that $c=ak^r+bk^s$ for some $0\leq r<r_a$ and $0\leq s<r_b$, which means that the congruence
\begin{align*}
l\cdot (a+b-c) & = l\cdot (a(1-k^r)+b(1-k^s)) \\
&= a\cdot l(1-k^r)+b\cdot l(1-k^s) 
\equiv 0 \mod n
\end{align*}
holds.
Furthermore we have 
\begin{align*}
e_{a,\alpha}+e_{b,\beta}-e_{c,\gamma} &= \frac{al+nq_{a,\alpha}}{\gcd(n,l)}+ \frac{bl+nq_{b,\beta}}{\gcd(n,l)}- \frac{cl+nq_{c,\gamma}}{\gcd(n,l)} \\
&= \frac{l\cdot (a+b-c)+n\cdot (q_{a,\alpha}+q_{b,\beta}-q_{c,\gamma})}{\gcd(n,l)}
\equiv 0 \mod \frac{n}{\gcd(n,l)} \,,
\end{align*}
finishing our claim. 
\end{proof}

\begin{corollary}\label{C:RootOf1Lem}
Let $a,b,c$ in $I$ be such that $S(a,b,c)\neq 0$. Then
$$
\bigg\{\frac{al}n\bigg\}
+
\bigg\{\frac{bl}n\bigg\}
-
\bigg\{\frac{cl}n\bigg\}
$$
is an integer.
\end{corollary}

\begin{proof}
By \eqref{E:frptalnZ}, 
$$
\frac n{\gcd(n,l)}
\bigg[
\bigg\{\frac{al}n\bigg\}
+
\bigg\{\frac{bl}n\bigg\}
-
\bigg\{\frac{cl}n\bigg\}
\bigg]
$$
is an integer. By \eqref{E:ea,alpha} and Lemma~\ref{RootOf1Lem}, we have 
$$
\frac n{\gcd(n,l)}
\bigg[
\bigg\{\frac{al}n\bigg\}
+
\bigg\{\frac{bl}n\bigg\}
-
\bigg\{\frac{cl}n\bigg\}
\bigg]
\equiv
 0 \mod \frac{n}{\gcd(n,l)}\,.
$$
The claim follows.
\end{proof}

\begin{lemma}
For any $a,b,c\in I$ and $0\leq \alpha <s_a, 0\leq \beta<s_b, 0\leq \gamma<r_c$. the specification
\begin{equation}\label{ProdCoeff}
\begin{cases} \dfrac{r_c S(a,b,c)}{\lcm(r_a,r_b)} & \text{if }  \alpha + \beta - \gamma \equiv \{al/n \}+\{bl/n \}-\{cl/n \}  \mod \gcd(s_a,s_b) \,,\\ 0 & otherwise \end{cases}
\end{equation}
gives the value of $\braket{\theta_{a,\alpha}\cdot \theta_{b,\beta},\theta_{c,\gamma}}$.
\end{lemma} 
\begin{proof}
First notice that $\theta_{a,\alpha}\theta_{b,\beta}(y^jx^i)\neq 0$ if and only if $j=\lcm(r_a,r_b)j'$ for some $0\leq j' < \dfrac{m}{\lcm(r_a,r_b)}=\gcd(s_a,s_b)$. Using this we have $\langle \theta_{a,\alpha}\cdot \theta_{b,\beta},\theta_{c,\gamma} \rangle=$
\begin{align*}
 &\frac{1}{nm} \sum_{j=0}^{m-1}\sum_{i=0}^{n-1} (\theta_{a,\alpha}\cdot \theta_{b,\beta})(y^jx^i) \overline{\theta_{c,\gamma}(y^jx^i)} \\
&= \frac{1}{nm} \sum_{j=1}^{\gcd(s_a,s_b)}\sum_{i=0}^{n-1} (\theta_{a,\alpha}\cdot \theta_{b,\beta})(y^{\lcm(r_a,r_b)j}x^i)\overline{\theta_{c,\gamma}(y^{\lcm(r_a,r_b)j}x^i)} \\
&= \frac{1}{nm} \sum_{j=1}^{\gcd(s_a,s_b)} \eta^{\lcm(r_a,r_b)(e_{a,\alpha}+e_{b,\beta}-e_{c,\gamma})j} \underbrace{\sum_{i=0}^{n-1} \sum_{s\in Ka} \sum_{r\in Kb} \sum_{t\in Kc} \omega^{(r+s-t)i}}_{n r_c S(a,b,c)} \\
&= \frac{r_c S(a,b,c)}{m} \sum_{j=1}^{\gcd(s_a,s_b)} \eta^{\lcm(r_a,r_b)(e_{a,\alpha}+e_{b,\beta}-e_{c,\gamma})j} \\
 &= \begin{cases} \dfrac{r_cS(a,b,c)\gcd(s_a,s_b)}{m} & \text{if } \tfrac{\gcd(n,l)}{n} \cdot(e_{a,\alpha} + e_{b,\beta} - e_{c,\gamma}) \equiv 0 \mod \gcd(s_a,s_b) \,,\\ 0 & otherwise \end{cases} \\
  &= \begin{cases} \dfrac{r_cS(a,b,c)}{\lcm(r_a,r_b)} & \text{if } \alpha + \beta - \gamma \equiv \{al/n \}+\{bl/n \}-\{cl/n \} \mod \gcd(s_a,s_b)\,, \\ 0 & otherwise\,, \end{cases} 
\end{align*}
where the penultimate equality holds because $\eta^{\lcm(r_a,r_b)(e_{a,\alpha}+e_{b,\beta}-e_{c,\gamma})}$ is a $\gcd(s_a,s_b)$-th root of unity by Lemma \ref{RootOf1Lem}.
\end{proof}

Equation~\ref{ProdCoeff} tells us that the product of two characters may be written explicitly as  
$$\theta_{a,\alpha}\cdot \theta_{b,\beta}=\sum_{c\in I}\sum_{\gamma} \frac{r_c S(a,b,c)}{\lcm(r_a,r_b)} \theta_{c,\gamma} $$
for each $a,b\in I, 0\leq \alpha <s_a$ and $0\leq \beta <s_b$, where the second summation is over all $0\leq \gamma< s_c$ such that the congruence
\begin{equation}\label{GammaCond}
\alpha+\beta -\gamma \equiv \{al/n \}+\{bl/n \}-\{cl/n \} \mod \gcd(s_a,s_b)
\end{equation}
holds.

\begin{lemma}
For each $a,b\in I$ and $0\leq \alpha <s_a$ and $0\leq \beta <s_b$ 
\begin{equation}\label{ProdCoeff2}
\theta_{a,\alpha}\cdot \theta_{b,\beta}=\frac{r_a}{\lcm(r_a,r_b)}\sum_{s\in Kb} \sum_\gamma \theta_{a+s,\gamma}\,,
\end{equation}
where the second summation is over all $\gamma$ which satisfy \eqref{GammaCond}.
\end{lemma}

\begin{proof}
Since $lk \equiv l \mod n$, it immediately follows that $lk^i\equiv l \mod n$ for any choice of $i$. Therefore, the value of $\{ak^il/n \}$ is independent of the choice of $i$. In other words \eqref{GammaCond} is unchanged by different choices of representatives for $a,b$ and $c$. Hence it suffices to show that for fixed $\gamma$, we have 
$$\sum_{c\in I} r_c S(a,b,c)\theta_{c,\gamma}=\sum_{s\in Kb} \tfrac{r_a}{r_{a+s}}\theta_{a+s,\gamma}\,.$$
To see why this is the case, recall that $S(a,b,c)\neq 0$ if and only if $c=a+bk^j$ for some $0\leq j<r_b$. Additionally, we have the following identity:
$$ \sum_{i=1}^{r_a}\sum_{j=1}^{r_b} (ak^i+bk^j)=\sum_{j=1}^{r_b} \sum_{i=1}^{r_a} (a+bk^{j-i})k^i= \sum_{j=1}^{r_b} \frac{r_a}{r_{a+bk^j}} \sum_{i=1}^{r_{a+bk^j}} (a+bk^j)k^i\,. $$ 
For fixed $0\leq r\leq r_b$, the number of times that any element of the form $(a+bk^r)k^s$ with $0\leq s\leq r_{a+bk^r}$ appears in the left hand side is exactly $r_{a+bk^r}S(a,b,c)$, while on the right hand side they appear $\sum_{j=1}^{r_b} {r_a}/{r_{a+bk^j}}$ many times. 
\end{proof}

\section{The character groups}

\subsection{Quasigroup semialgebras}\label{SS:QgpSmalg}

Let $Q$ be a quasigroup, say a character quasigroup for a finite group $G$. Let $\mathbb NQ$ be the free $\mathbb N$-semimodule over $Q$, with the product $\cdot$ defined by semilinear extension of the quasigroup multiplication on $Q$. In particular, if $Q$ is a group, then $\mathbb NQ$ forms the fragment of the integral group algebra where all the coefficients are nonnegative.

A multisubset $M$ of $Q$, where each element $q$ of $Q$ appears in $M$ with muliplicity $m_q$, is represented in $\mathbb NQ$ as 
$$
M=\sum_{q\in Q}m_qq\,.
$$ 
Additionally if $S$ and $T$ are multisubsets of $Q$, then $m_q(S,T)$ will denote the multiplicity of $q$ in the product $S\cdot T$.  

\begin{lemma}\label{EquivDef}
Let $(X,w,\alpha)$ be a weighted quasigroup. A quasigroup $Q$ covers $(X,w,\alpha)$ if and only if there are $|X|$ many disjoint subsets $\{S_x \}_{x\in X}$ such that $|S_x|=w(x)$ for each $x\in X$ and 
\begin{equation}\label{CovEq}
 \sum_{q\in S_z} m_q(S_x,S_y)=\alpha(x,y,z)
\end{equation} 
for each $x,y,z$, 
\end{lemma}
\begin{proof}
Given the sets $S_x$, define $f:Q\rightarrow X$ by $S_x\mapsto x$, and conversely given the covering $f$ define $S_x:=f^{-1}(x)$.  
\end{proof}

Given a weighted quasigroup $(X,w,\alpha)$, the covering quasigroups can be difficult to work with in general, so we would like to restrict ourselves to studying a subset of the covering quasigroups which are more well-behaved. Let $Q$ be a  covering quasigroup $Q$, together with the sets $\{S_x\}_{x\in X}$ guaranteed by Lemma ~\ref{EquivDef}. We say $Q$ is a \emph{balanced covering} of $(X,w,\alpha)$ if $m_q(S_x,S_y)$ does not depend on the choice of $q\in S_z$ for all $x,y,z\in X$.

\begin{remark}
If one does not want to refer explicitly to the notation of Lemma~\ref{EquivDef} in the definition of a balanced covering, they may equivalently take it as a covering such that 
$$\forall\ x,y\in X\,,\ \alpha(x,y,z)=\alpha(x,y,z') $$
whenever $f^{-1}\{z \}=f^{-1}\{z' \}$. 
\end{remark}

Note that in a balanced covering 
\begin{equation}
m_q(S_x,S_y)=\frac{\alpha(x,y,z)}{w(z)} 
\end{equation} 
for all $x,y,z\in X$ and $q\in S_z$. In particular, when the covering is balanced, \eqref{CovEq} becomes 
\begin{align*}
S_x\cdot S_y &= \sum_{q\in Q} m_q(S_x,S_y)q 
= \sum_{z\in X}\sum_{q\in S_z} m_q(S_x,S_y) q \\
&= \sum_{z\in X}\sum_{q\in S_z} \frac{\alpha(x,y,z)}{w(z)}q 
= \sum_{z\in X} \frac{\alpha(x,y,z)}{w(z)}S_z \,.
\end{align*}
In other words, the sets $\set{S_x}_{x\in X}$ carry a nice algebraic structure when the covering is balanced. This property can be used to compute character groups, by building a presentation for them.

Let $G$ be a finite group with set $\tilde{\Gamma}(G)=\Set{ \varphi_1,\dots, \varphi_t }$ of irreducible characters,  satisfying 
$$ \varphi_i \cdot \varphi_j = \sum_{k=1}^t \langle \varphi_i \cdot \varphi_j , \varphi_k \rangle \varphi_k\,. $$ 
Summarizing the properties of balanced coverings of weighted quasigroups to the case where $Q$ is a character quasigroup, we have the following result. 

\begin{corollary}\label{CovChar}
A quasigroup $Q$ is a balanced cover of $G$ if and only if there exist subsets $\Set{S_i}_{1\leq i\leq t}$ of $Q$ such that $|S_i|=\varphi_i(1)^2$ and 
\begin{equation}
S_i\cdot S_j = \sum_{k=1}^t \left(\frac{\varphi_i(1)\varphi_j(1)\langle \varphi_i\cdot \varphi_j , \varphi_k \rangle}{\varphi_k(1)} \right) S_k
\end{equation}
for each $1\leq i,j\leq t$.
\end{corollary}

Specializing to the case of metacyclic groups, we obtain the following.

\begin{corollary}\label{MetaCov1}
A quasigroup $Q$ is a balanced cover of
$$ G=\braket{x,y \ | \ y^m=x^l, x^n=1, x^y=x^k}  $$
 if and only if there exist disjoint subsets $S_{a,\alpha}$ of size $r_a^2$ for each $a\in I$ and $0\leq \alpha < s_a$ such that
\begin{equation}\label{MetCovReq}
S_{a,\alpha}\cdot S_{b,\beta}= gcd(r_a,r_b)\sum_{a\in I}\sum_\gamma S(a,b,c)S_{c,\gamma}\,,
\end{equation}
where $\gamma$ satisfies \eqref{GammaCond}.
\end{corollary}

\begin{proof}
Corollary~\ref{CovChar} yields 
$$ S_{a,\alpha}\cdot S_{b,\beta}= \sum_{c\in I}\sum_\gamma \frac{r_ar_br_c}{r_c\text{lcm}(r_a,r_b)}S_{c,\gamma} = \gcd(r_a,r_b)\sum_{c\in I}\sum_\gamma S(a,b,c)S_{c,\gamma}\,, $$
where $\gamma$ satisfies \eqref{GammaCond}.
\end{proof}

\begin{corollary}\label{MetaCov2}
A quasigroup $Q$ is a balanced cover of $M(k.l,m,n)$
 if and only if there exist disjoint subsets $S_{a,\alpha}$ of size $r_a^2$ for each $a\in I$ and $0\leq \alpha < s_a$ such that
\begin{equation}\label{MetCovReq2}
S_{a,\alpha}\cdot S_{b,\beta}= r_a\gcd(r_a,r_b)\sum_{i=0}^{r_b-1}\sum_\gamma \frac{S_{a+bk^i,\gamma}}{r_{a+bk^i}}\,,
\end{equation}
where $\gamma$ satisfies \eqref{GammaCond}.
\end{corollary}

\begin{proof}
Corollary~\ref{CovChar} yields 
$$ S_{a,\alpha}\cdot S_{b,\beta}= \sum_{i=0}^{r_b-1}\sum_\gamma \frac{r_ar_br_a}{r_{a+bk^i}\lcm(r_a,r_b)}S_{a+bk^i,\gamma} = r_a\gcd(r_a,r_b)\sum_{i=0}^{r_b-1}\sum_\gamma \frac{S_{a+bk^i,\gamma}}{r_{a+bk^i}}\,, $$
where $\gamma$ satisfies \eqref{GammaCond}.
\end{proof}

Note that for each $0\leq j<s_b$,
$$
S_{0,s_bj} \cdot S_{b,\beta} = \sum_{i=0}^{r_b-1} \sum_\gamma \tfrac{S_{bk^i,\gamma}}{r_{bk^i}} = \sum_\gamma S_{b,\gamma} =S_{b,\beta} \,,
$$
and so multiplication by $S_{0,s_bj}$ fixes $S_{b,\beta}$ setwise, allowing us to write 
$$ S_{b,\beta} = \sum_{j=0}^{r_b-1} S_{0,s_bj} T_{b,\beta} $$
for each $b\in I$, where $T_{b,\beta}$ is a set of size $r_b$ for each $0\leq \beta <s_b$.

\subsection{The dihedral groups}

The dihedral group $D_d$ of degree $d$ and order $2d$ is the metacyclic group $M(-1,0,2,d)$. Balanced character groups of dihedral groups were studied in \cite{StevensonSmith}. We now reproduce one of the results from that work \cite[Th.~4.1]{StevensonSmith} within the present context. For simplicity, we omit the references to generalized quaternion groups that also appeared in \cite[Th.~4.1]{StevensonSmith}.

\begin{theorem}\label{even covering}
Let $Q$ be a finite quasigroup of order $2d$, where $d=2k$ is even. Then $Q$ functions as a character quasigroup of $D_d$
if and only if there exist elements $q_1,q_2,q_3,q_4\in Q$ and subsets $S_i\subseteq Q$ for $0<i<k$, each of size $4$, along with $S_0=\Set{q_1,q_2}$ and $S_k=\Set{q_3,q_4}$, such that the conditions
\begin{enumerate}
\item[$(0)$]
$Q$ is the disjoint union of the sets $S_0,\dots,S_k;$
\item 
$S_0 \cup S_k \cong \mathbb{Z}/2 \oplus \mathbb{Z}/2$ as a subgroup of $Q\,;$
\item
$ 
\begin{cases}
(\mathrm a)\quad
q_jS_i=S_i=S_iq_j &\mbox{for }0<i<k \mbox{ and }j=1,2\,,\\
(\mathrm b)\quad
q_jS_i=S_{k-i}=S_iq_j &\mbox{for }0<i<k \mbox{ and }j=3,4\,;
\end{cases}
$
\item $ S_j\cdot S_i=S_i\cdot S_j=S_{k-i}\cdot S_{k-j} \mbox{ for } 0<i,j,i+j<k\,;$
\item 
for $0<i,j<k$,
$$
S_i \cdot S_j=
\begin{cases} 
4S_0 + 4S_k & \mbox{if } \ i=j=k/2, \\
2S_{i+j} + 4S_0 & \mbox{if } \ i=j  \neq k/2, \\ 
4S_k + 2S_{i-j} & \mbox{if } \ i+j=k \mbox{ and } i\neq j, \\
2S_{i+j} + 2S_{i-j} & \mbox{otherwise} \\
\end{cases}
$$
\end{enumerate} 
are satisfied. In condition $(4)$, the indices $h$ on the sets $S_h$ are interpreted as residues modulo $d$, with the convention that $S_h=S_{-h}$.
\end{theorem}

\begin{proof}
By Corollary~\ref{MetaCov1}, a quasigroup $Q$ is a balanced covering of $D_d$ if and only if there exist disjoint subsets $S_{a,\alpha}$ of size $r_a^2$ for each $a\in I$ and $0\leq \alpha <s_a$ such that \eqref{MetaCov1} is satisfied.  When $r_a=r_b=1$, Equation~\ref{GammaCond} becomes $\alpha +\beta \equiv \gamma \mod 2$, and so in this case 
$$ S_{a,\alpha}\cdot S_{b.\beta}=S_{a+b,\alpha+\beta}\,. $$
With $\set{q_1,q_2,q_3,q_4}=\set{S_{0,0},S_{0,1},S_{k,0},S_{k,1}}$, condition $(1)$ is clearly satisfied. 
Conditions $(2)$ and $(3)$ are trivially verified, and when $r_a=r_b=2$, Equation~\ref{GammaCond} becomes $\alpha +\beta \equiv \gamma \mod 1$ for any $c\in I$, and in particular is satisfied for all $0\leq \gamma<s_c$. Therefore in this case equation \ref{GammaCond} becomes
$$ S_{a,0}\cdot S_{b,\beta } = 2\sum_{c=0}^k \sum_\gamma S(a,b,c) S_{c,\gamma}\,. $$
Easy case analyses reveal that $S(a,b,0)=4$ when $a=b$, and $S(a,b,k)=4$ when $a+b=k$, while $S(a,b,a+b)=S(a,b,a-b)=2$ in all other cases, yielding condition $(4)$. 
\end{proof}

Using Corollary~\ref{MetaCov2}, we can get an upgraded theorem for the balanced covering groups of the dihedral groups. 

\begin{theorem}
A quasigroup $Q$ is a balanced cover of $D_d$ if and only if there exist disjoint subsets $S_{a,\alpha}$ of size $r_a^2$ for each $a\in I$ and $0\leq \alpha <s_a$ such that 
$$ S_{a,\alpha}\cdot S_{b,\beta} = r_a\gcd(r_a,r_b)\sum_{i=0}^{r_b-1} \sum_\gamma \frac{S_{a+b(-1)^i}}{r_{a+b(-1)^i}}\,, $$
where $\gamma$ satisfies \eqref{GammaCond}.
\end{theorem}

\subsection{The dicyclic groups}

The dicyclic group Dic$_l$ is defined to be the metacyclic group $M(-1,l,2,2l)$ \cite[Ex.~8.2.3]{Scott}. We specialize our general results to this case. 

\begin{lemma}\label{FloorLem}
If $S(a,b,c)\neq 0$, then $\{a/2 \}+\{b/2 \}-\{c/2 \}$ is an integer.
\end{lemma}

\begin{proof}
This is an instance of Corollary~\ref{C:RootOf1Lem}, noting that $n=2l$ here.
\end{proof}

\begin{theorem}\label{DicyclicCov1}
A quasigroup $Q$ is a balanced cover of 
$$ G=\braket{x,y \ | \ y^2=x^l, x^{2l}=1, y^{-1}xy=x^{-1}} $$
if and only if there exist disjoint subsets $S_a \subseteq Q$ for $0\leq a \leq l$ such that $S_a$ is of size $r_a$ and: 
\begin{enumerate} 
\item[(a)] $\langle S_0, S_l\rangle \cong \begin{cases} 
\mathbb{Z}/_2 \oplus \mathbb{Z}/_2 & \text{if $l$ is even;} \\
\mathbb{Z}/_4 & \text{if $l$ is odd;} \\
\end{cases}$
\item[(b)] For all $0\leq a,b\leq l$ with $(r_a,r_b)\neq (1,1)$,
 \begin{equation}\label{DicyclicEq}
 S_a\cdot S_b= \tfrac{4\gcd(r_a,r_b)}{r_b} \sum_{i=0}^{r_b-1} \tfrac{S_{a+b(-1)^i}}{r_{a+b(-1)^i}}\,,
\end{equation}  
\end{enumerate}
where we impose $S_a=S_{ak^i}$ for all $i\geq 0$ and the subscript is taken modulo $n$. 
\end{theorem}

\begin{proof}
First suppose that $Q$ covers $G$. By Corollary~\ref{MetaCov1} there exist disjoint subsets $S_{a,\alpha}$ of size $r_a^2$ for each $a\in I$ and $0\leq \alpha <s_a$ which satisfy the required equations. If we take $S_0=S_{0,0}+S_{0,1}$ and $S_l=S_{l,0}+S_{l,1}$, these will clearly satisfy $(a)$, and upon dividing the right hand side of  \eqref{MetCovReq} by $s_as_b$, we obtain condition $(b)$. 

Conversely, suppose we are given disjoint subsets $S_a$ of size $r_a$ satisfying conditions $(a)$ and $(b)$. Define each of the following: 
$$ S_{0,0}=e, \qquad S_{0,1}=S_0, \qquad S_{l,0}=S_l, \qquad S_{l,1}=S_0S_l, \qquad $$
and $S_a=S_{a,0}$ for each $0<a<l$. By Lemma~\ref{FloorLem}, equation \eqref{GammaCond} is always satisfied when $(r_a,r_b)\neq (1,1)$, and so \eqref{MetCovReq} must hold. 
\end{proof}

For the duration of this section we will fix $S_0=x_0+y_0$ and $S_l=x_l+y_l$, which by previous observations means we may write $S_a=(x_0+y_0)(x_a+y_a)$ for some $x_a,y_a\in Q$ for each $0<a<l$. Without loss of generality we are able to chose $x_0$ such that $x_0y_0=y_0$. In particular, if $Q$ is a group, then $x_0=1$. 

By Theorem~\ref{DicyclicCov1}, $\set{x_0,y_0,x_l,y_l}$ must form a copy of either $\mathbb{Z}/_2 \oplus \mathbb{Z}/_2$ if $l$ is even or $\mathbb{Z}/_4$ if $l$ is odd. We already have a partial Cayley table 
\begin{center}
\large
 \begin{tabular}{c | c c c c}
    $\cdot$ & $x_0$ & $y_0$ & $x_l$ & $y_l$  \\
    \cline{1-5}
    $x_0$ & $x_0$ & $y_0$ & $x_l$ & $y_l$  \\
    $y_0$ & $y_0$ & $x_0$ &  &  \\
    $x_l$ & $x_l$ & &  &  \\
    $y_l$ & $y_l$ & &  &  \\
	   \end{tabular}
\end{center}
\renewcommand{\arraystretch}{1}	
filled out, which immediately allows us to say $y_0\cdot x_l=y_l$ and $y_0\cdot y_l=x_l$. Moreover, if $x_l^2=x_0$, then the collection is isomorphic to $\mathbb{Z}/_2\oplus \mathbb{Z}/_2$, and if $x_l^2=y_0$, then it is isomorphic to $\mathbb{Z}/_4$. Thus, we require $x_l^2=y_0^l$, for $Q$ to be a character quasigroup for Dic$_l$.

\begin{lemma}
For each $1\leq i \leq l-1$, the relation
$$ (x_0+y_0)(x_i+y_i)(x_0+y_0)=2(x_0+y_0)(x_i+y_i) $$ 
holds
\end{lemma}

\begin{proof}
By equation \ref{DicyclicEq} the quartets must satisfy $S_i\cdot S_0=2S_i$, in particular 
$$ (x_0+y_0)(x_i+y_i)(x_0+y_0)=2(x_0+y_0)(x_i+y_i) $$ 
\end{proof}

\begin{lemma}\label{ThreeLem0}
Let $Q$ be a character quasigroup for Dic$_l$. Then exactly one of the following holds:
\begin{enumerate}
\item[(a)] $x_1^2,y_1^2\in \set{x_0,y_0}$ and $S_2=x_1y_1+y_1x_1$\,;
\item[(b)] $x_1^2, x_1y_1\in \set{x_0,y_0}$ and $S_2=y_1x_1+y_1^2$\,;
\item[(c)] $x_1^2,y_1x_1\in\set{x_0,y_0}$ and $S_2=x_1y_1+y_1^2$\,;
\item[(d)] $x_1y_1,y_1x_1\in\set{x_0.y_0}$ and $S_2=x_1^2+y_1^2$\,.
\end{enumerate}
\end{lemma}
\begin{proof}
We will compute $S_1^2$ in two ways. On the one hand, using the previous lemma, we have
$$ S_1^2 = (x_0+y_0)(x_1+y_1)(x_0+y_0)(x_1+y_1)=2(x_0+y_0)(x_1^2+x_1y_1+y_1x_1+y_1^2)\,, $$
which by Theorem~\ref{DicyclicCov1} is equal to $4e+4q+2S_2$. Thus exactly two of $x_1^2,x_1y_1,y_1x_1$ and $y_1^2$ belong to the set $\set{x_0,y_0}$ and the remaining two elements belong to the set $S_2$.
 
If both $x_1^2$ and $y_1^2$ belong to $\{x_0,y_0 \}$ then we are in case (a), and if exactly one of them is in the set then without loss of generality we may assume that it is $x_1^2$, leading to cases (b) and (c). Finally if $x_1^2$ and $y_1^2$ both do not belong to the set $\set{x_0,y_0}$ then we get case (d).  
\end{proof}

\begin{remark}
If $Q$ is a group then case (c) of Lemma~\ref{ThreeLem0} is no longer possible, since if $y_1x_1=y_0^j$ for some $j\in \set{0,1}$, then $y_1=y_0^ix^{-1}=y_0^{i+j}x_1$, which contradicts the size of $S_1$. In case (b) we can further say that $x_1y_1=y_0$ and $x_1^2=x_0$, since they clearly cannot be equal, and if $x_1y_1=x_0$ and $x_1^2=y_0$, then $y_1=x_1^{-1}x_0=y_0x_0$, again contradicting the size of $S_1$. Finally, in case (d) we can say $x_1y_1=y_1x_1$, since $x_1y_1=x_0$ if and only if $y_1x_1=x_0$. 
\end{remark}

\begin{lemma}\label{ThreeLem}
Let $Q$ be a character quasigroup for Dic$_n$. Then exactly one of the following holds: 
\begin{enumerate}
\item[(0)] $x_1y_0=x_1$\,; 
\item $x_1y_0=y_1$\,;
\item$x_1y_0=y_0x_1$ and $y_1y_0=y_0y_1$\,; 
\item $x_1y_0=y_0y_1$\,.
\end{enumerate}
\end{lemma}
\begin{proof}
Since $S_1y_0=S_1$ we have 
$$ (x_0+y_0)(x_1+y_1)q=(x_0+y_0)(x_1+y_1)\,, $$
from which it follows that $x_1y_0\in \{x_1,y_1,y_0x_1,y_0y_1 \}$. 
\end{proof}

\begin{remark}
If $Q$ is a group, case $(0)$ in Lemma~\ref{ThreeLem} is impossible, since it would imply that $x_0=y_0$. 
\end{remark}

\begin{theorem}
Let $G$ be a character group for Dic$_l$, such that $x_1^2,y_1^2\not\in \{x_0,y_0\}$. Then $G\cong \mathbb{Z}/_{2l}\times \mathbb{Z}/_2$ or $D_{2l}$ if $l$ is even, and $\mathbb{Z}/_{4l}$ if $l$ is odd. 
\end{theorem}

\begin{proof}
First note that since $S_2=(x_0+y_0)(x_1^2+y_1^2)$, we can prove via induction that $S_i=(x_0+y_0)(x_1^i+y_1^i)$ for each $1\leq i<l$ and $S_l=x_1^l+y_1^l$. Indeed, by \eqref{DicyclicEq}, we have $S_1\cdot S_i = S_{i-1}+S_{i+1}$ and 
\begin{align*}
S_1 \cdot S_i &= (e+q)(x_1+y_1)(e+q)(x_1^i+y_1^i) \\
&= 2(e+q)(x_1+y_1)(x_1^i+y_1^i) \\
&= 2(e+q)(x_1^{i+1}+x_1y_1^i+y_1x_1^i+y_1^{i+1}) \\
&= 2(e+q)(x_1^{i-1}+y_1^{i-1}+x_1^{i+1}+y_1^{i+1}) \,.
\end{align*}
This immediately implies that $x_1^{2l}=x_l^2=y_0^l$, meaning if $l$ is odd then $G$ is necessarily cyclic, so from this point on we will assume that $l$ is even.  

By Lemma~\ref{ThreeLem} we have three possibilities to check. 
\begin{itemize}
\item
If $y_1=x_1y_0$, then $y_1^2=y_1x_1y_0\in \set{x_0,y_0}$, which contradicts our assumption. 
\item
If $x_1y_0=y_0x_1$ and $y_1y_0=y_0y_1$, then $G$ has presentation 
$$ \braket{x_1,y_1,y_0 \ | \ x_1^{2l}=y_1^{2l}=y_0^2=1, x_1y_0=y_0x_1, y_1y_0=y_0y_1, x_1y_1=y_1x_1=y_0^i} $$
where $i\in \set{0,1}$.  In either case this group is isomorphic to $\mathbb{Z}/_{2l} \oplus \mathbb{Z}/_2$. 
\item
Finally if $x_1y_0=y_0y_1$ then $G$ has presentation 
\begin{align*}
&\braket{x_1,y_1,y_0 \ | \ x_1^{2l}=y_1^{2l}=y_0^2=1, x_1y_0=y_0y_1, x_1y_1=y_1x_1=y_0^i } \\
 &\rule{40mm}{0mm}
 = \braket{x_1,y_0 \ | \ x_1^{2l}=y_0^2=1, y_0^{-1}x_1y_0=x_1^{-1}y_0^i}
\end{align*}
where $i\in \set{0,1}$. Clearly if $i=0$ this is the dihedral group $D_{2l}$, and if $i=1$ then $x_1^2=y_0$, leading to the contradiction $G\cong \mathbb{Z}/_4$. 
\end{itemize}
This completes the proof.
\end{proof}

\begin{theorem}
If $l$ is odd, then no character group of Dic$_l$ satisfies condition $(b)$ of Lemma~\ref{ThreeLem0}. and if $l$ is even with character group $G$ such that $x_1^2=e, x_1y_1=y_0$ and $y_1^2\not\in \set{x_0,y_0}$, then $G\cong D_{2l}$. 
\end{theorem}

\begin{proof}
Note that in this case $S_2=(x_0+y_0)(y_1x_1+y_1^2)$. Using an induction argument similar to that of the previous theorem, we can show that $S_i=(x_0+y_0)(y_1^{i-1}x_1+y_1^i)$ for $1\leq i<l$ and $S_l=y_1^{l-1}x_1+y_1^l$. So we have $y_1^{2l}=y_0^l$ and $(y_1^{l-1}x_1)^2=y_0^l$, and 
$$ y_0^l = y_1^{l-1}x_1y_1^{l-1}x_1 = (x_1y_0)^{l-1}x_1(x_1y_0)^{l-1}x_1 = x_1(y_0x_1)^{l-1}(x_1y_0)^{l-1}x_1=1\,, $$
which shows that $l$ must be even. In this case $G$ has presentation 
$$ \braket{x_1,y_0 \ | \ x_1^2=y_0^2=1, (x_1y_0)^{2l}=1} $$
which yields $D_{2l}$.
\end{proof}

\begin{theorem}
Let $G$ be a character group for Dic$_l$, such that $x_1^2,y_1^2\in \set{e,q}$. Then $l$ is either $1$ or $2$ modulo $4$, and $G\cong\braket{x,y \ | \ x^4=y^2=(xy)^l=1, x^2y=yx^2 }$.
\end{theorem}

\begin{proof}
Using a similar inductive procedure as before, one can show that: 
$$ x_i=\begin{cases} (x_1y_1)^{\tfrac{i}{2}} & \text{ if $i$ is even}; \\
(x_1y_1)^{\tfrac{i-1}{2}}x_1 & \text{ if $i$ is odd};
\end{cases} 
\quad
\mbox{and}
\quad
y_i = \begin{cases}  
(y_1x_1)^{\tfrac{i}{2}} & \text{ if $i$ is even}; \\
(y_1x_1)^{\tfrac{i-1}{2}}y_1 & \text{ if $i$ is odd}. \\
\end{cases}
 $$
By Lemma~\ref{ThreeLem}, we have three cases to consider. 
\begin{itemize}
\item
First if $x_1y_0=y_1$ then 
$$ x_1y_1=x_1^2y_0\in \set{x_0,y_0}\,, $$
which cannot happen. 
\item
If $x_1y_0=y_0y_1$, then $y_1^2=y_0x_1^2y_0=x_1^2$, and moreover both of $x_1^2$ and $y_1^2$ must be $1$, since otherwise
$$ y_1=y_0x_1y_0=y_0y_0^2y_0=1\,. $$
If $l$ is even, then 
$$ y_0 = x_ly_l = (x_1y_1)^{\frac{l}{2}}(y_1x_1)^{\frac{l}{2}}=1\,, $$
a clear contradiction; and if $l$ is odd then 
$$ y_0 = x_l^2 = (x_1y_1)^{\frac{l-1}{2}}x_1 (x_1y_1)^{\frac{l-1}{2}}x_1 = 1\,, $$
again a clear contradiction. 
\item
Finally assume that $x_1y_0=y_0x_1$ and $y_1y_0=y_0y_1$, and moreover first suppose that $l$ is even, so that 
$x_l=(x_1y_1)^{\tfrac{l}{2}}$
and
$y_l=(y_1x_1)^{\tfrac{l}{2}}$.
This  implies that 
 $$ y_0=x_ly_l=(x_1y_1)^{\frac{l}{2}}(y_1x_1)^{\frac{l}{2}}=q^{\frac{l}{2}(i+j)}\,, $$
 which means that 
 $\tfrac{l}{2}(i+j)$
 must be $1$ modulo $2$. This is impossible if $l$ is $0$ modulo $4$, so we conclude that $l$ must be $2$ modulo $4$. We can also conclude that $i\neq j$, thus without loss of generality we may assume that $i=1$ and $j=0$. Therefore, the presentation for $G$ becomes 
  \begin{align*}
 &\braket{x_1,y_1,y_0 \ | \ x_1^2=y_0, y_1^2=y_0^2=1, \ x_1y_0=y_0x_1, \ y_1y_0=y_0y_1, \ (x_1y_1)^l=1 } \\
 &\rule{50mm}{0mm}
  =\braket{x_1,y_1 \ | \ x_1^4=y_1^2=(x_1y_1)^l=1, x_1^2y_1=y_1x_1^2}\,.
 \end{align*}
  On the other hand, if $l$ is odd, then 
 $x_l=(x_1y_1)^{\tfrac{l-1}{2}}x_1$ 
 and
$y_1=(y_1x_1)^{\tfrac{l-1}{2}}y_1$. 
Therefore 
 $$ y_0= x_l^2 = ((x_1y_1)^{\frac{l-1}{2}}x_1)^2 = y_0^{\frac{l+1}{2}(i+j)}\,, $$
which is impossible if $l$ is $3$ modulo $4$ or $i=j$, and so the same conclusion follows.
\end{itemize}
This completes the proof. 
\end{proof}

\subsection{Split metacyclic groups}

A metacyclic group is split when $l=0$, in this case being nothing more than a semidirect product of two cyclic groups. We simply have 
$e_{a,\alpha}=q_{a,\alpha}=\alpha$
 for all $a\in I$. Therefore \eqref{GammaCond} reduces to $\alpha+\beta\equiv \gamma \mod \gcd(s_a,s_b)$.  We will set 
$$ G=:\mathbb{Z}/_n \rtimes \mathbb{Z}/_m=\braket{x,y \ | \ y^m=1, b^n=1, x^y=x^k}  $$
throughout this section.

\begin{theorem}\label{MetCovClas}
Let $Q$ be a quasigroup of order $nm$. If $Q$ has an element $q$ of order $m$, and disjoint subsets $T_a$ of size $r_a$ for each $a\in I$ such that
$qT_a=T_aq$ 
 and:
 \begin{enumerate}
\item $\bigcup_{\gamma=1}^m \bigcup_{a\in I} q^\gamma T_a = Q$\,;
\item $T_a \cdot T_b = \sum_{c\in I} S(a,b,c) T_c$
for all $a,b\in I$\,,
\end{enumerate}
then $Q$ is a balanced cover of $G$. 
\end{theorem}

\begin{proof}
Suppose that we are given elements $\Set{q_0,\dots, q_{m-1}}$ and disjoint subsets $T_a$ of size $r_a$ for each $a\in I$ which satisfy conditions $(1)$ and $(2)$. We define 
$$S_{a,\alpha}=\left(\sum_{i=0}^{r_a-1} q_{s_ai+\alpha}\right)T_a$$ 
for each $a\in I$ and $0\leq \alpha <r_a-1$.
\end{proof}

\begin{corollary}
The group 
$\mathbb{Z}/_n\rtimes_\sigma \mathbb{Z}/_m$
acts as a balanced cover of 
$\mathbb{Z}/_n\rtimes \mathbb{Z}/_m$
for any group homomorphism
$\sigma\colon\mathbb{Z}/_m\to\mathrm{Aut}(\mathbb{Z}/_n)$. 
\end{corollary}

\begin{proof}
Suppose $\sigma\colon1\mapsto[1\mapsto k^b]$ for some $b$. Define 
$$ T_a= \Set{(ak^i,0): 0\leq i<r_a} $$
for each $a\in I$, and $q_j=(0,j)$ for each $0\leq j<m-1$. Then the conditions $(1)$ and $(2)$ of Theorem~\ref{MetCovClas} are clearly satisfied. 
\end{proof}

\end{document}